%% file: 000_main.tex
\documentclass{amsart}
\usepackage[utf8]{inputenc}

\input{000_setup}
\title{Loci of the Brocard Points over\\Selected Triangle Families}
\author{Ronaldo Garcia} 
\author{Dan Reznik} 
\date{September, 2020}

\begin{document}

\maketitle

\input{005_abstract}

\section{Introduction}
\label{sec:intro}
\input{010_intro}

\section{Circle-Mounted Triangles}
\label{sec:circle-mounted}
\input{020_circle_mounted}

\section{Homothetic Pair Loci}
\label{sec:homot}
\input{030_homothetic}

\section{Brocard Triangles}
\label{sec:broc-tris}
\input{040_brocard_triangles}

\section{Conclusion}
\label{sec:conclusion}
\input{100_conclusion}

\input{110_ack}

\appendix
\section{Brocard Circle of the Homothetic Pair}
\label{app:homot-broc-circle}
\input{210_app_homot_broc_circle}

\bibliographystyle{maa}
\bibliography{references,authors_rgk}

\end{document}

%% file: 000_setup.tex
\usepackage{graphics}
\usepackage{thmtools}
\usepackage{wasysym}
\usepackage[T1]{fontenc}    

\usepackage{amsthm}
\usepackage{amsbsy,amsmath,amssymb,amscd,amsfonts}
\usepackage[pagebackref=true]{hyperref}

\usepackage{graphicx,float,latexsym,color}
\usepackage[font={scriptsize,it}]{caption}
\usepackage{subcaption}

\usepackage{makecell}

\usepackage[dvipsnames]{xcolor}

\newtheorem*{theorem*}{Theorem}
\newtheorem{observation}{Observation}
\newtheorem{proposition}{Proposition}

\newtheorem{corollary}{Corollary}

\theoremstyle{remark}
\newtheorem{remark}{Remark}
\theoremstyle{definition}
\newtheorem{definition}{Definition}

\hypersetup{
    pdftoolbar=true,        
    pdfmenubar=true,        
    pdffitwindow=false,     
    pdfstartview={FitH},    
    colorlinks=true,       
    linkcolor=OliveGreen,          
    citecolor=blue,        
    filecolor=black,      
    urlcolor=red           
}

\usepackage{lineno}

\usepackage{comment}
\usepackage{microtype}
\usepackage{footnote}

%% file: 005_abstract.tex
\begin{abstract}
We study the loci of the Brocard points over two selected families of triangles: (i) 2 vertices fixed on a circumference and a third one which sweeps it, (ii) Poncelet 3-periodics in the homothetic ellipse pair. Loci obtained include circles, ellipses, and teardrop-like curves. We derive expressions for both curves and their areas. We also study the locus of the vertices of Brocard triangles over the homothetic and Brocard-poristic Poncelet families. 

\vskip .3cm
\noindent\textbf{Keywords} Poncelet, Brocard, Homothetic, Porism, Locus
\vskip .3cm
\noindent \textbf{MSC} {53A04 \and 51M04 \and 51N20}
\end{abstract}

%% file: 010_intro.tex
The Brocard points $\Omega_1$ and $\Omega_2$ \cite{johnson1960} are two unique interior points through which Cevians at a special angle $\omega$ to the sides concur; see Figure~\ref{fig:brocard-basic}. In triangle parlance, they are known as a {\em bicentric pair} \cite{kimberling2003-bicentric}. Here we study the loci of $\Omega_1,\Omega_2$ over a family of triangles over two selected families of triangles, to be sure:

\begin{figure}
    \centering
    \includegraphics[width=.6\textwidth]{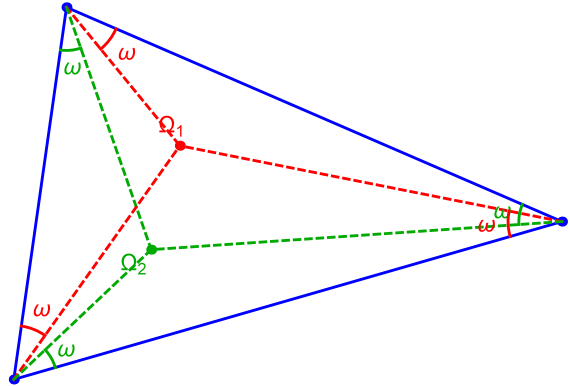}
    \caption{The Brocard Points $\Omega_1$ (resp. $\Omega_2$) are where sides of a triangle concur when rotated about each vertex by the Brocard angle $\omega$. When sides are traversed and rotated clockwise (resp. counterclockwise), one obtains $\Omega_1$ (resp. $\Omega_2$).}
    \label{fig:brocard-basic}
\end{figure}
 
 \begin{itemize}
     \item Circle-Mounted: two vertices fixed at points on a circumference (or ellipse boundary) and a third one which slides over it.
     \item Homothetic Pair: Poncelet 3-Periodics interscribed in a homothetic, concentric pair of ellipses. This family conserves the Brocard angle \cite{reznik2020-similarityII}.
 \end{itemize}
 
 Under the above Brocard loci include a circle, teardrop-curves, and variations. We derive explicit expressions for the loci and study their area ratios with respect to the generating curves.
 
 We also study the locus of vertices of the Brocard Triangles\footnote{The 7th Brocard was invented by one author and Peter Moses during this research: its vertices are the intersections of cevians through $X_3$ with the Brocard Circle. The inspiration was the 2nd Brocard whose vertices lie at intersections of cevians through $X_6$ with said circle \cite[Second Brocard Triangle]{mw}.} (there are 7 of them, see \cite{gibert2020-brocard}) under two families: homothetic and Brocard-poristic (whose Brocard points are stationary \cite{bradley2007-brocard}). In the former case the only the vertices of the First Brocard is an ellipse, and in the latter 4 out of 7 describe circles, both results stemming directly from known homotheties of said triangles.

\subsection{Related Work}

Ferréol describes a construction for a right strophoid as the locus of the orthocenter of a family of triangles with two fixed vertices and a third one revolving on a circumference \cite{ferreol-strophoid}. Odehnal has studied loci of triangle centers for the poristic triangle family \cite{odehnal2011-poristic}; similarly, Pamfilos proves properties of the family of triangles with fixed 9-point and circumcircle \cite{pamfilos2020}. We have studied the loci of triangle centers for 3-periodics in the elliptic billiard (confocal Poncelet pair), identifying a few centers whose loci are elliptic  \cite{garcia2020-ellipses}, and built and interactive app to visualize loci of centers of several triangle families \cite{darlan2020-app}.

Below when referring to triangle centers we adopt Kimberling's $X_k$ notation \cite{etc}.

%% file: 020_circle_mounted.tex
Let a family of triangles be defined with two vertices $V_1,V_2$ stationary with respect to a circle of radius $a$ (say centered at the origin) and a third one $P(t)$ which executes on revolution over the the circumference, $P(t)=a[\cos{t},\sin{t}]$.

Referring to Figure~\ref{fig:brocard-drops-mixed} (left):

\begin{proposition}
The locus of $\Omega_1$ (resp. $\Omega_2$) with $V_1=(0,0)$ and $V_2=(0,a)$ is a circle of radius $\frac{a}{3}$ (resp. a teardrop curve) of area $\frac{\pi a^2}{9}$ 
(resp. $\frac{2\pi a^2}{9}$).
\end{proposition}

\begin{proof}

In this case we have that:
\[\Omega_1(t)=a\left[ {\frac {\cos t }{5-4\,\sin t }},{
\frac { 2-a\sin t  }{5-4\,\sin t }}\right]
\]

\[\Omega_2(t)=a\left[\frac { 2\cos t-\sin 2t   }{ 5-4\,
\sin t },{\frac {    2 \sin t +\cos
 2t  }{5-4\,\sin t}}\right]\]

\end{proof}

\begin{remark}
The above loci intersect at $a[\pm\sqrt{3}/6,1/2]$; along with $V_2=(0,a)$ they define an equilateral. This stems from the fact that when $P(t)=a[\pm\sqrt{3}/2,1/2]$, $V_1V_2P(t)$ is equilateral and the two Brocard points coincide at the Barycenter $X_2$.
\end{remark}
Referring to Figure~\ref{fig:brocard-drops-mixed} (right):

\begin{proposition}
The locus of $\Omega_1$ and $\Omega_2$ with $V_1=(0,a)$ and $V_2=(a,0)$ are a pair of inversely-identical\footnote{Identical modulo inverse similarity \cite[Inversely Similar]{mw}.}, skewed teardrop shapes with the following equations and areas:
\end{proposition}

\begin{proof}
We have that:
\[ \Omega_1(t)=a\left[{\frac {       \sin^2 t   +\cos
 t-\sin t    }{ \left( \sin
 t-2 \right) \cos t-2\,\sin   t
 +3}}, \frac {    1-\cos t }{
 \left( \sin t-2 \right) \cos t-2\,
\sin t+3}\right]
\]

\[ \Omega_2(t)=a
\left[{\frac {    1-\sin t   }{ \left( \cos
 t-2 \right) \sin t-2\,\cos   t
   +3}},{\frac { \left(   \cos^2 t  
 -\cos t+\sin t \right) }{ \left( 
\sin t-2 \right) \cos t-2\,\sin
 t+3}}\right]\]
 
 Defining $D(x,y)=(y,x)$, the reflection abott the diagonal,
 it follows that $\Omega_2(t)=(D\circ \Omega_1)(t-\frac{\pi}{2}).$

\end{proof}

In Figure~\ref{fig:brocard-circle} we show the shape of the locus varies in a complicated way when $V_2=(0,a)$ and $V_1=(x,0)$, with $0{\leq}x{\leq}a$.




Referring to Figure~\ref{fig:brocard-antipodes} (left), Robert Ferréol has kindly contributed \cite{ferreol2020-private-larme}:

\begin{proposition}
With $V_1=(-a,0)$ and $V_2=(a,0)$, the loci of the Brocards are a pair of inversely-identical teardrop shapes whose areas are $\pi a^2/\sqrt{5}$. The one with a cusp on $V_1$ is given by the following quartic:

\begin{equation}
     x^4-2 x^3+2 y^2 x^2+2 x-2 y^2 x-1+y^4+4 y^2=0 
\end{equation}
\end{proposition}

\begin{proof}
\[ \Omega_1(t)= \left[{\frac {-a\cos 2t -8\,a\cos t+a}{
\cos 2t -9}},{\frac {-2\,a\sin 2t -4
\,a\sin t}{\cos 2t  -9}}\right]
\]
\[ \Omega_2(t)= \left[{\frac {-8\,a\cos t+a\cos 2t -a}{
\cos \left( 2\,u \right) -9}},{\frac {2\,a\sin 2t -4
\,a\sin t}{\cos 2t -9}}.\right]\]
Let $R(x,y)=(-x,y$. Then $\Omega_2(t)=(R\circ\Omega_1)(t)$.
The implicit form of $\Omega_2$ is given by 
\[ B_2(x,y)=a^2(a^2-2 a x-4 y^2)+2 a x (x^2+y^2)  -(x^2+y^2)^2=0.\]
Analogously,
$B_1(x,y)=B_2(-x,y)=0$ is the implicit form of $\Omega_1$.

The area of the region bounded by $\Omega_i$ is given by $\frac{1}{2}\int_{\Omega_i} xdy-ydx$. It follows that $A(\Omega_i)=\frac{\sqrt{5} \pi a^2}{5}$.  
\end{proof}

Figure~\ref{fig:brocard-antipodes} (right) depicts the loci of $\Omega_1$ and $\Omega_2$ with $P(t)$ on an ellipse with semi-axes $(a,b)$ and with $V_1=(-a,0)$ and $V_2=(a,0)$. These are a pair of symmetric teardrop curves whose complicated parametric equations we omit.

\begin{figure}
    \centering
    \includegraphics[width=\textwidth]{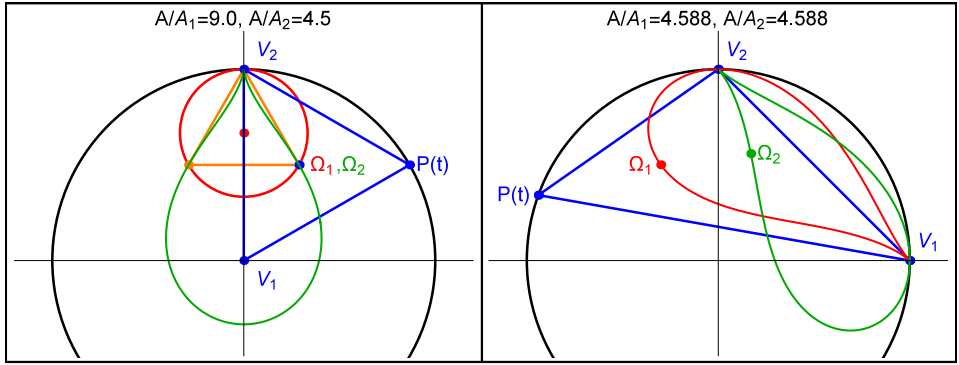}
    \caption{\textbf{Left:} $V_1$ and $V_2$ are affixed to the center and top vertex of the unit circle and a third one $P(t)$ revolves around the circumference. The locus of the Brocard points $\Omega_1,\Omega_2$ are a circle (red) and a teardrop (green) whose areas are 1/9 and 2/9 that of the generating circle. The sample triangle (blue) shown is equilateral, so the two Brocard points coincide. Notice the curves' two intersections along with the top vertex form an equilateral (orange). \textbf{Right:} $V_1,V_2$ are now placed at the left and top vertices of the unit circle. The Brocard points of the family describe to inversely-similar teardrop shapes. \href{https://youtu.be/Ms8jC9yOKU4}{Video}, \href{https://bit.ly/3jX6FII}{Live}}
    \label{fig:brocard-drops-mixed}
\end{figure}

\begin{figure}
    \centering
    \includegraphics[width=.925\textwidth]{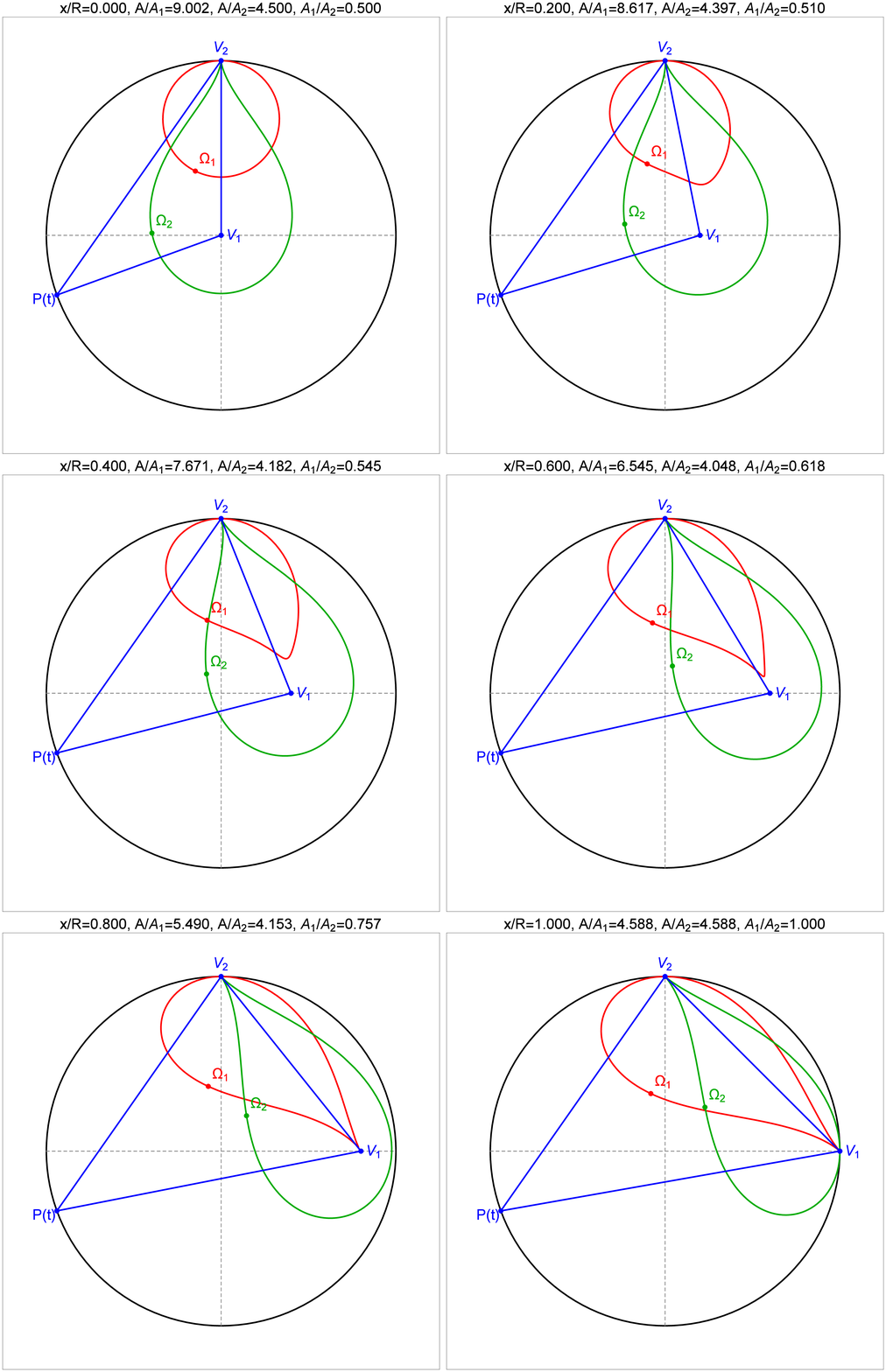}
    \caption{Shape of loci of Brocard Point $\Omega_1$ (red) and $\Omega_2$  (green) for $V_2$ fixed at $(0,1)$, as $V_1$ moves from the origin along the $x$ axis toward $(1,0)$. The loci are obtained over a complete revolution of $P(t)$ on a unit circle (black). \textbf{Top left:} $V_1=(0,0)$, the locus of $\Omega_1$ (resp. $\Omega_2)$ is a perfect circle (resp. a teardrop curve) of $1/9$ (resp. $2/9$) the area of the external. \textbf{Bottom right:} when $V_1=(a,0)$ the two loci are inversely-similar copies of each other, whose areas are $1/\sqrt{5}{\simeq}0.447$ that of the circle. \href{https://youtu.be/Ms8jC9yOKU4}{Video}, \href{https://bit.ly/3jX6FII}{Live}}
    \label{fig:brocard-circle}
\end{figure}

\begin{figure}
    \centering
    \includegraphics[width=\textwidth]{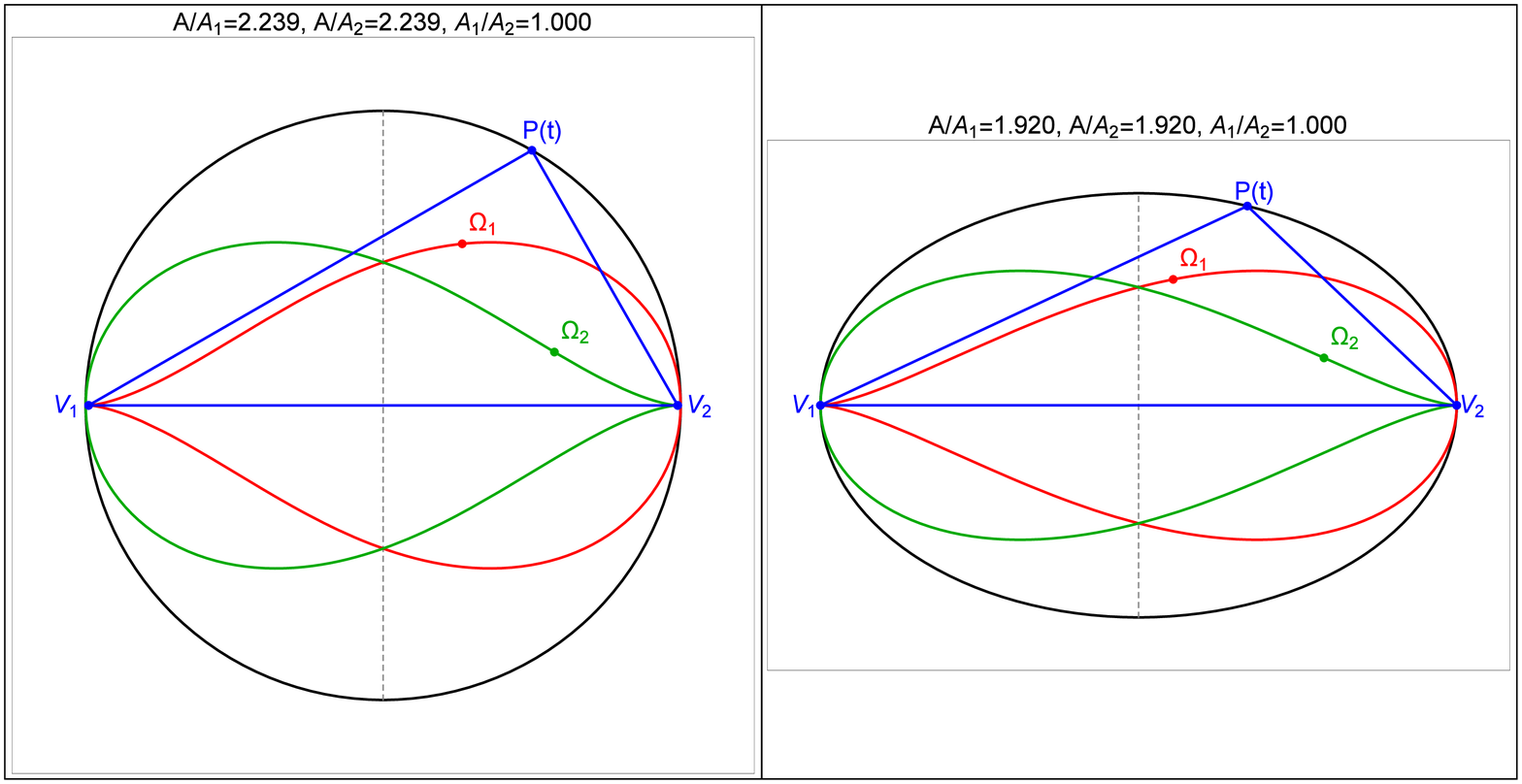}
    \caption{\textbf{Left:} With antipodal $V_1$ and $V_2$ and $P(t)$ revolving on the circumference, the loci of the Brocards are symmetric teardrops whose area are $1/\sqrt{5}$ that of the circle. \textbf{Right:}. With $V_1,V_2$ at the major vertices of an ellipse of axes $(a,b)$, and $P(t)$ revolving on its boundary, the the Brocard loci (red and green) are still symmetric (though stretched) teardrop shapes. In this case $a/b=1.5$. \href{https://bit.ly/3m53uQT}{Live}}
    \label{fig:brocard-antipodes}
\end{figure}

\begin{proposition}
The locus of $\Omega_1$ and $\Omega_2$ with $V_1=(x_1,0)$, $|x_1|\leq a$,  $V_2=(-a,0)$ and $V_3=(a\cos t,a\sin t) $ are a pair of singular  teardrop curves with the following   areas:
\begin{align*} A_1=& \,\frac { 4\left( x_1+a \right) ^{2}{a}^{5}\pi}{ \left( 3\,{a}^{
2}+x_1^{2} \right) ^{2}\sqrt {4\,{a}^{2}+x_1^{2}}}
\\
A_2=&\frac { \left( 2\,{a}^{2}-ax_1+x_1^{2} \right)  \left( x_1+a \right) ^{3}{a}^{2}\pi}{ \left( 3\,{a}^{2}+x_1^{2}
 \right) ^{2}\sqrt {4\,{a}^{2}+x_1^{2}}}
\end{align*}
\end{proposition}

When $x_1=a$, the ratio of $A1$ and $A2$ by the area of the circle $a^2\pi$ both reduce to $1/\sqrt{5}{\simeq}0.4472$.

\begin{proof} 
The above is obtained with direct integration and simplification with a computer algebra system (CAS).
\end{proof}

%% file: 030_homothetic.tex
Consider an origin-centered ellipse with semi-axes $(a,b)$ and an internal concentric, axis-aligned one with semi-axes $(a',b')=(a/2,b/2)$. This pair is associated with a 3-periodic Poncelet porism since the condition $a'/a+b'/b=1$ is satisfied \cite{georgiev2012-poncelet}.

\begin{figure}
    \centering
\includegraphics[width=\textwidth]{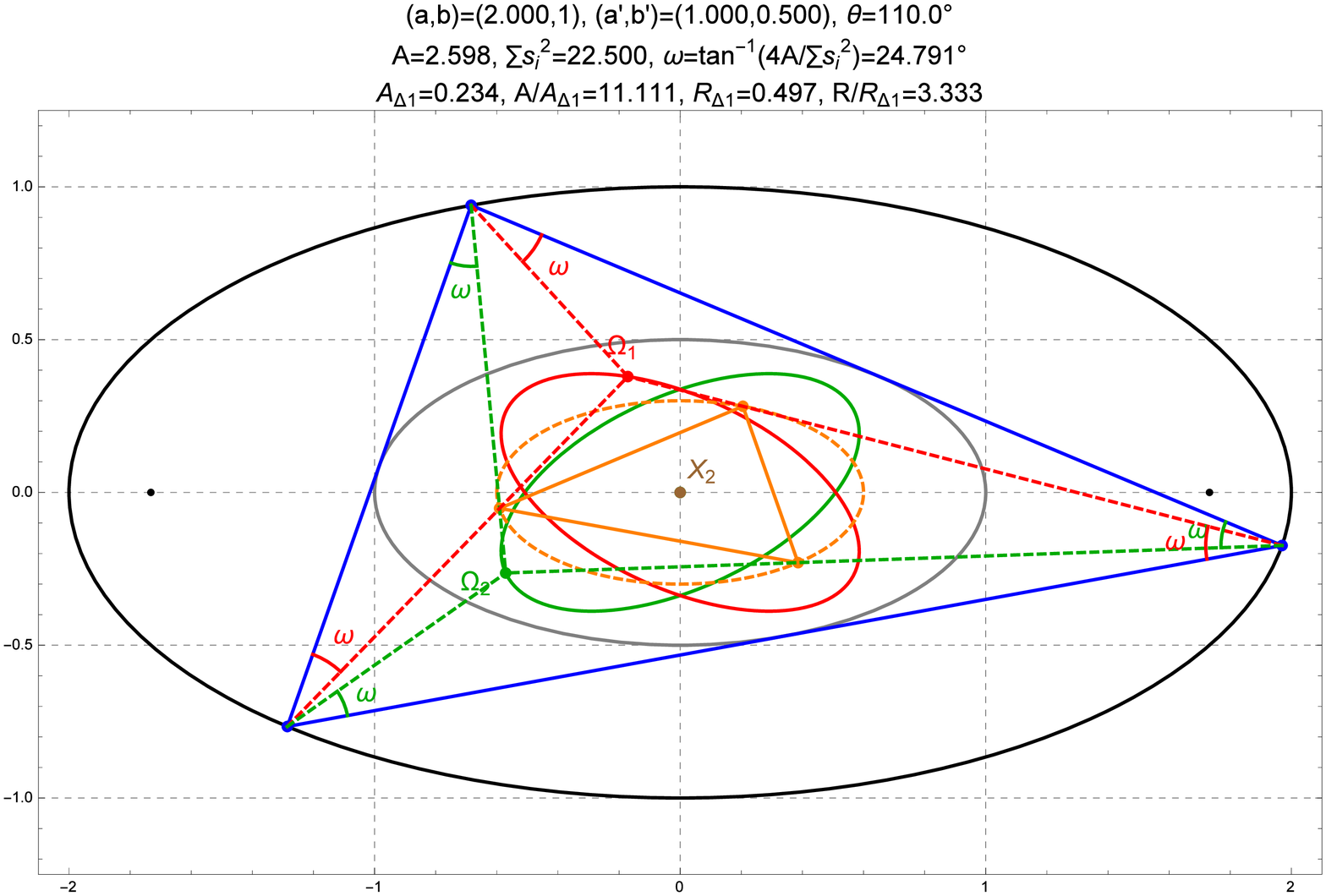}
\caption{The 1d Poncelet 3-periodic family interscribed in the homothetic pair conserves sum of squared sidelengths, area, and Brocard angle $\omega$ \cite{reznik2020-similarityII}. The loci of the two Brocard points $\Omega_1$ and $\Omega_2$ are tilted ellipses (red and green) of aspect ratio equal to those in the pair \href{https://youtu.be/2fvGd8wioZY}{Video}. The locus (dashed orange) of the vertices of the first Brocard triangle (orange) is an axis-aligned ellipse also homothetic to the pair.\href{https://youtu.be/13i3JGY-fK4}{Video}, \href{https://bit.ly/3iaISog}{App}}
    \label{fig:homot-loci}
\end{figure}

Referring to Figure~\ref{fig:homot-loci}:

\begin{proposition}
The loci of the Brocard points over 3-periodics in the homothetic pair are ellipses $E_1$ and $E_2$ which are reflected images of each other about either the $x$ or $y$ axis. Furthermore these are concentric and homothetic to the ellipses in the pair.
\end{proposition}

\begin{proof}
The loci are given by
	
	\[E_1(x,y)= {\frac { \left( 7\,{a}^{4}+6\,{a}^{2}{b}^{2}+3\,{b}^{4} \right) {x}^{2
}}{{a}^{2} \left( a^2-b^2 \right) ^{2}   }}+{\frac {
 \left( 3\,{a}^{4}+6\,{a}^{2}{b}^{2}+7\,{b}^{4} \right) {y}^{2}}{{b}^{
2} \left( {a}^{2}-  {b}^{2} \right)^2 }}- \,{\frac {
4\sqrt {3} \left( {a}^{2}+{b}^{2} \right) xy}{ab \left( {a}^{2}-{b}^{2}
 \right) }}-1
	\]
	\[E_2(x,y)=	 {\frac { \left( 7\,{a}^{4}+6\,{a}^{2}{b}^{2}+3\,{b}^{4} \right) {x}^{2
}}{{a}^{2} \left( a^2-b^2 \right) ^{2}   }}+{\frac {
 \left( 3\,{a}^{4}+6\,{a}^{2}{b}^{2}+7\,{b}^{4} \right) {y}^{2}}{{b}^{
2} \left( {a}^{2}- {b}^{2} \right)^2 }}+ \,{\frac {
4\sqrt {3} \left( {a}^{2}+{b}^{2} \right) xy}{ab \left( {a}^{2}-{b}^{2}
 \right) }}-1
\]	
	The angle $\theta$ between the axes of ellipses $E_1$  and $E_2$ is given by
	\[\tan\theta = \frac{4\sqrt{3}(a^2+b^2) ab}{3a^4+2a^2b^2+3b^4}.\]
\end{proof}

In no other concentric Poncelet pairs studied so far (poristic, incircle, inellipse, dual, confocal) is the locus of either Brocard point an ellipse. 

\begin{remark}
At $a/b=\sqrt{5}$ the elliptic loci of the Brocard points over the homothetic family are internally tangent to the inner ellipse.
\end{remark}

\begin{remark}
At $a/b{\simeq}3.8$ the elliptic loci of the Brocard points over the homothetic family intersect the y axis at $b/2$, i.e., at the top vertex of the caustic.
\end{remark}

In Appendix~\ref{app:homot-broc-circle} we derive a few properties of the Brocard circle with respect to the homothetic family.

%% file: 040_brocard_triangles.tex
Consider a triangle $T=P_1{P_2}{P_3}$ with Brocard points $\Omega_1$ and $\Omega_2$. Referring to Figure~\ref{fig:first-broc-tri}:

\begin{definition}[First Brocard Triangle]
The vertices $P_1'$, $P_2'$, $P_3'$ of the First Brocard Triangle $T_1$ are defined as follows: $P_1'$ (resp. $P_2'$, $P_3'$) is the intersection of $P_2{\Omega_1}$ (resp. $P_3{\Omega_1}$, $P_1{\Omega_1}$) with $P_3{\Omega_2}$ (resp. $P_1{\Omega_2}$, $P_2{\Omega_2}$).
\end{definition}

Know properties of the $T_1$ include that (i) it is inversely similar to $T$, (ii) its barycenter $X_2$ coincides with that of the reference triangle, and (iii) its vertices are concyclic with $\Omega_1$, $\Omega_2$, $X_3$, and $X_6$ on the Brocard circle \cite[Brocard Circle]{mw}, whose center is $X_{182}$.

\begin{figure}
    \centering
    \includegraphics[width=.66\textwidth]{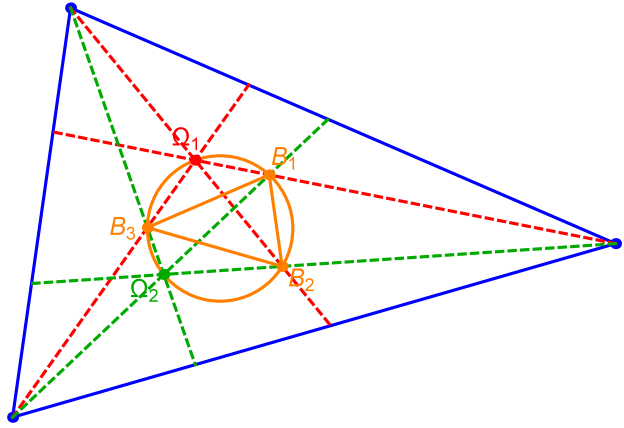}
    \caption{Construction for the First Brocard Triangle (orange) taken from \cite[First Brocard Triangle]{mw}. It is inversely similar to the reference one (blue), and their barycenters $X_2$ are common. Its vertices $B_1,B_2,B_3$ are concyclic with the Brocard points $\Omega_1$ and $\Omega_2$ on the Brocard circle (orange).}
    \label{fig:first-broc-tri}
\end{figure}

\subsection{Homothetic Pair}

Referring to Figure~\ref{fig:homot-loci}:

\begin{proposition}
Over 3-periodics in the homothetic pair, the locus of the vertices of $T_1$ is an axis-aligned, concentric ellipse, homothetic to the ones in the pair and interior to the caustic. Its axes are given by: 
\[ a'=\frac{a(a^2-b^2)}{2(a^2+b^2)},\;\;\; b'=\frac{b(a^2-b^2)}{2(a^2+b^2)}\]
\end{proposition}

\begin{proof}
The locus must be an ellipse since $T_1$ is inversely similar to the 3-periodics whose vertices are inscribed in an ellipse and their barycenters coincide. A vertex of the Brocard triangle is parametrized by

\[ \frac{x^2}{a'^2}+\frac{y^2}{b'^2} = 1 \]
 
It can be shown $a'<a/2$ and $b'<b/2$ therefore the locus is interior to the caustic, i.e., the stationary Steiner inellipse.
\end{proof}

Since homothetic 3-periodics conserve area \cite{reznik2020-similarityII}, so must $T_1$ (inversely similar). Its area can be computed explicitly:

\begin{remark}
Over 3-periodics in the homothetic pair, the area of $T_1$ is invariant and given by
\[ {\frac {3\sqrt{3} ab\left( {a}^{2}-{b}^{2} \right) ^{2}  }{16
 \left( {a}^{2}+{b}^{2} \right) ^{2}}}
\]
\end{remark}

\begin{corollary}
The similarity ratio of homothetic 3-periodics to the $T_1$ is invariant and given by
\[ \frac{ 2(a^2+b^2)}{ a^2-b^2}\cdot\]
\label{cor:sim-ratio}
\end{corollary}

\subsection{Vertices of the $T_1$ over the homothetic family}

Let the ${P_1}{P_2}{P_3}$ be the vertices of a 3-periodic in the homothetic pair, and ${P'_1}{P'_2}{P'_3}$ those of $T_1$. These are given by:

\begin{align*}
P_1'=& k_1 R_x P_2\\
P_2'=&k_1 R_x P_3\\
P_3'=&k_1 R_x P_1
%
\end{align*}

where $k_1=\frac{a^2-b^2}{2(a^2+b^2)}$ and $R_x(x,y)=(-x,y)$ is a reflection.

\subsection{Brocard Porism}

This is a 3-periodic family inscribed in a fixed circumcircle and a fixed inellipse, known as the Brocard Inellipse \cite[Brocard Inellipse]{mw}. A remarkable property of this family is that the Brocard points are stationary at the foci of said inellipse \cite{bradley2007-brocard}.

Given the axes $(a,b)$ of the latter, we have shown elsewhere that the circumcenter $X_3$, circumradius $R$, and (conserved) Brocard angle $\omega$ of the family are given by \cite{reznik2020-similarityII}:

\begin{equation}
X_3=[0,-\frac{c\delta_1}{b}],\;\;\;R= \frac{2a^2}{b}, \;\;\; \cot\omega=  \frac{\delta_1}{b} \\
 \label{eqn:broc-circumcircle}
\end{equation}
where $\delta_1=\sqrt{4a^2-b^2}$.

Let $T_k$ denote the nth-Brocard triangle, $k=1,\cdots,7$, as defined in  \cite{gibert2020-brocard}. Referring to Figure~\ref{fig:broc-por-tris}

\begin{remark}
For $k=1,2,5,7$, the locus of vertices of $T_k$ trace out the Brocard circle.
\end{remark}

\begin{figure}
    \centering
    \includegraphics[width=\textwidth]{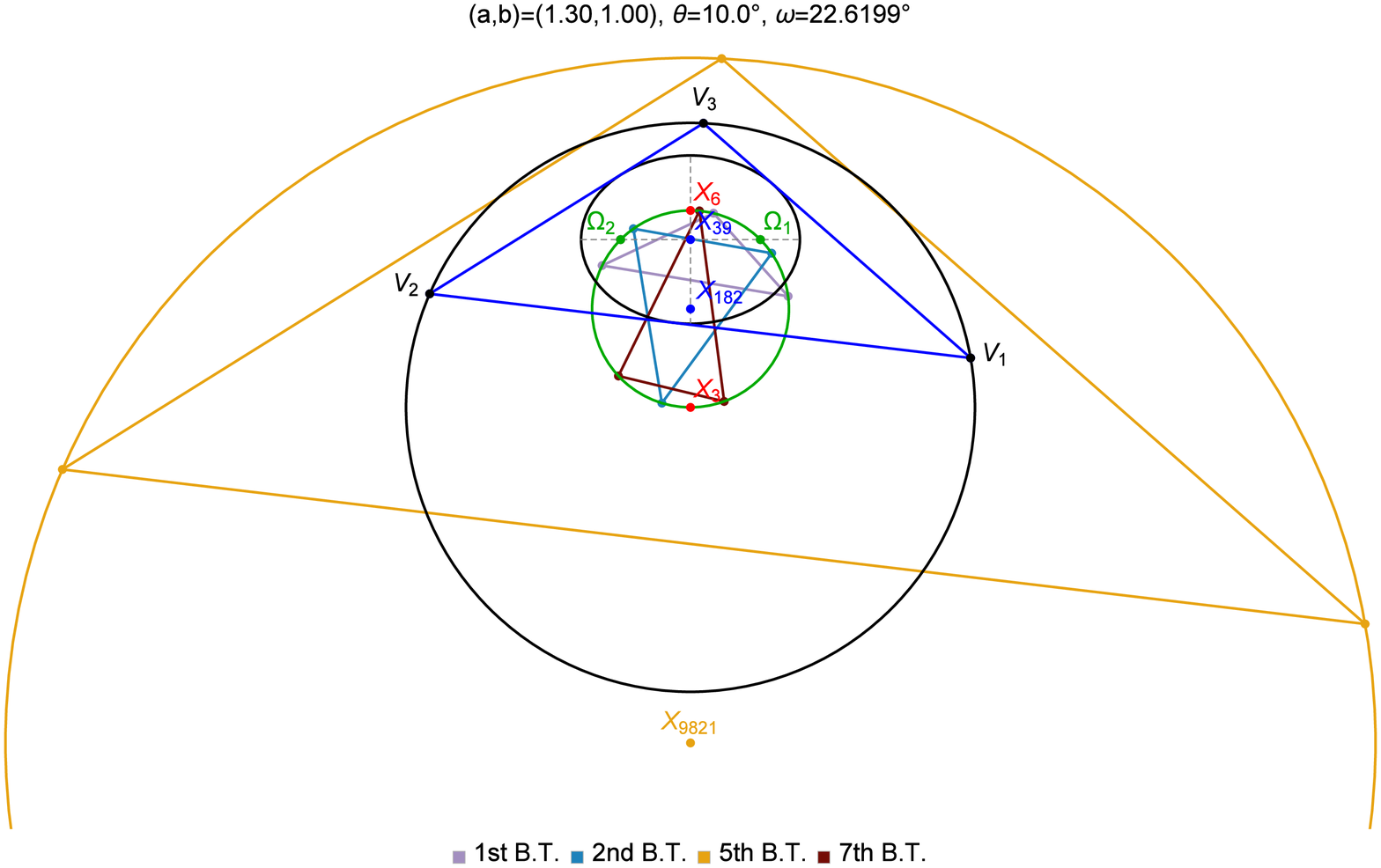}
    \caption{A 3-periodic (blue) in the Brocard Porism is shown inscribed in a circle and circumscribed about the Brocard inellipse (both black). Its Brocard points $\Omega_1$ and $\Omega_2$ are stationary at the foci of said inellipse. The First, Second, and Seventh Brocard Triangle are shown inscribed in the Brocard circle (green). The Fifth Brocard Triangle (orange)is homothetic about the 3-periodic and therefore its locus will also be a circle (not shown). \href{https://youtu.be/_bK-BCQv24A}{Video}, \href{https://bit.ly/32GFvQu}{Live}}
    \label{fig:broc-por-tris}
\end{figure}

This stems from the fact that by construction, $T_1$, $T_2$, $T_7$ are inscribed in the Brocard circle (their circumcenter is $X_{182}$ and $T_5$ is homothetic to the reference one and its circumcenter is $X_{9821}$ \cite{etc}. For no other Poncelet families and/or Brocard triangle combinations have we been able to identify conic loci for the Brocard Triangle vertices.

 As before, $\Omega_1,\Omega_2$ denote the Brocard points of the 3-periodics. Let $\Omega_1^j$ (resp. $\Omega_2^j$) denote the first (resp. second) Brocard point of $T_j$. Also let $X_i^j$ denote the $X_i$ center of $T_j$. The following observations are about the 3-periodic family in the Brocard porism.

\begin{observation}
$\Omega_1^2$ and $\Omega_2^2$ are stationary. Furthermore the triples  $(\Omega_1,X_6^2, \Omega_1^2)$ and $(\Omega_2,X_6^2,\Omega_2^2)$ are each collinear.
\end{observation}

Remarkably, the above implies a ``russian doll'' nesting of Brocard porisms composed of the original family and then successive Second Brocards, each inhabiting its own private porism, see this \href{https://youtu.be/Z3YlEbCFbnA}{video}. 

\begin{observation}
Both $\Omega_1^1$ and $\Omega_2^1$
move along the same circle $C_1$.
\end{observation}

\begin{observation}
Both $\Omega_1^6$ and $\Omega_2^6$
move along the same circle $C_6$.
\end{observation}

\begin{observation}
The locus of $\Omega_1^4$ and $\Omega_2^4$ are two distinct circles $C_4$ and $C_4'$.
\end{observation}

%% file: 100_conclusion.tex
A few questions are posed to the reader:

\begin{itemize}
    \item Are there other Poncelet ellipse pairs, concentric or not, whose Brocard points of their 3-periodics trace out ellipses? (so far he have found none)
    \item In the homothetic pair are there other derived triangles (besides the First Brocard), non-homothetic to the 3-periodics, whose vertices trace out conics?
\end{itemize}

A list of animations of some of the results above appears on Table~\ref{tab:videos}. 

\begin{table}[H]
\scriptsize
\begin{tabular}{|c|c|l|l|}
\hline
Exp & Video & App & Title \\
\hline
01 & \href{http://youtu.be/Ms8jC9yOKU4}{*} & \href{https://bit.ly/3jX6FII}{*} & 
\makecell[lt]{Loci of Brocard Points on Circle-Mounted\\ Triangles: Center-Top} \\
02 & n/a & \href{https://bit.ly/3m53uQT}{*} & 
\makecell[lt]{Loci of Brocard Points on Circle-Mounted\\ Triangles: Left-Right} \\
03 & \href{https://youtu.be/13i3JGY-fK4}{*} & \href{https://bit.ly/3iaISog}{*} & \makecell[lt]{Loci of the Brocard Points and Vertices of\\ 1st Brocard Triangle over the  Homothetic Family} \\
04 & \href{https://youtu.be/_bK-BCQv24A}{*} & \href{https://bit.ly/32GFvQu}{*} & \makecell[lt]{Circular locus of Brocard Triangles, $T_j$, $j=1,2,5,7$} \\
05 & \href{https://youtu.be/Z3YlEbCFbnA}{*} & n/a & \makecell[lt]{Russian-Doll Nesting of Second Brocard Triangles} \\
\hline
\end{tabular}
\caption{Experimental animations. Click on the * to see it as a {YouTube} video and/or a browser-based simulation.}
\label{tab:videos}
\end{table}

%% file: 110_ack.tex
\noindent We would like to thank Robert Ferréol for his Mathcurve portal (it inspired this work) as well as early derivations of loci. Mark Helman has kindly helped us with experiments and research around the Brocard loci. Bernard Gibert and Peter Moses for helping us usher the 7th Brocard Triangle into existence. The first author is fellow of CNPq and coordinator of Project PRONEX/ CNPq/ FAPEG 2017 10 26 7000 508.

%% file: 210_app_homot_broc_circle.tex
Recall the Brocard circle is the circumcircle of the $T_1$ \cite[Brocard Circle]{mw}.

\begin{corollary} Over 3-periodics in the homothetic pair, the ratio of areas of the circumcircle of 3-periodics to that of the Brocard circle is invariant and given by
\[  \frac{4(a^2+b^2)^2}{(a^2-b^2)^2} \cdot\]
\end{corollary}

\begin{proof}
The result follows from Corollary~\ref{cor:sim-ratio}. Straightforward calculations yield that the circumradius of 3-periodics is given by: 
\[ R^2= -\frac{(a^2-b^2)^3 \cos(6 t)}{32a^2b^2}+\frac{(a^2+b^2)(a^4+14a^2b^2+b^4)}{32a^2b^2}.\]
\end{proof}

\begin{proposition}
Over the homothetic family, the Brocard circle is given by:

\[ E(x,y)= \,{\frac { 4\left( {a}^{2}+{b}^{2} \right) ^{2}{x}^{2}}{{a}^{2}
 \left( a^2-b^2 \right) ^{2}  }}+ \,{\frac {4 \left( 
{a}^{2}+{b}^{2} \right) ^{2}{y}^{2}}{{b}^{2} \left( a^2-b^2 \right) ^{2}}}-1
\]
\end{proposition}